\documentclass[letterpaper]{article} 
\usepackage[utf8]{inputenc}
\usepackage{amsmath,amssymb}
\usepackage{hyperref}
\hypersetup{
    colorlinks=true,
    citecolor=blue,
}

\usepackage{float}
\usepackage{graphicx}
\usepackage{import}

\usepackage{amsthm}
\newtheorem{thm}{Theorem}
\newtheorem{lem}[thm]{Lemma}
\theoremstyle{remark}
\newtheorem*{remark}{Remark}
\theoremstyle{definition}

\numberwithin{thm}{section}

\usepackage{newpxmath}
\usepackage{newpxtext}

\usepackage{parskip}

\usepackage{empheq, xcolor}
\definecolor{myblue}{rgb}{.8, .8. 1}

\renewcommand{\d}{\delta}

\newcommand{\e}{\varepsilon}

\renewcommand{\l}{\lambda}

\newcommand{\s}{\sigma}

\newcommand{\cN}{{\mathcal N}}

\newcommand{\cS}{{\mathcal S}}

\newcommand{\bP}{{\mathbb P}}
\newcommand{\bE}{{\mathbb E}}
\newcommand{\bR}{{\mathbb R}}
\newcommand{\bC}{{\mathbb C}}

\renewcommand{\i}{\infty}

\DeclareMathOperator{\supp}{supp}
\DeclareMathOperator{\dist}{dist}
\DeclareMathOperator{\Var}{Var}
\DeclareMathOperator{\Sparse}{Sparse}

\title{Invertibility of Sparse Complex Gaussian Matrices}
\author{Edward Zeng \\ UC Berkeley \\ \href{mailto:edwardzeng@berkeley.edu}{edwardzeng@berkeley.edu}}

\begin{document}
\maketitle
\begin{abstract}
    Let $\s_n(\cdot)$ denote the least singular value of a $n \times n$ matrix. It is well-known that
    \[
        \bP[\s_n(A) \le \e] \le \e n
    \]
    if $A$ is drawn from the real Ginibre ensemble of $n \times n$ matrices and
    \[
        \bP[\s_n(A) \le \e] \le \e^2 n^2
    \]
    if $A$ is drawn from the complex Ginibre ensemble. In this paper, we will show a similar phenomenon occurs for sparse random matrices.
\end{abstract}

\section{Introduction} \label{intro}
Let $\s_n(\cdot)$ denote the least singular value of a $n \times n$ matrix. It is well-known that
\[
    \bP[\s_n(A) \le \e] \le \e n
\]
if $A$ is drawn from the real Ginibre ensemble of $n \times n$ matrices and
\[
    \bP[\s_n(A) \le \e] \le \e^2 n^2
\]
if $A$ is drawn from the complex Ginibre ensemble. (For proofs, see \cite{Ede88}.) In this paper, we will show a similar phenomenon occurs for sparse random matrices.

Let $\d_{p}$ denote the Bernoulli random variable
\[
    \bP[\d_{p} = 1] = p, \quad \bP[\d_{p} = 0] = 1 - p.
\]
In \cite{BaRu17} Basak and Rudelson prove the following result:
\begin{thm}
    \label{ru}
    Let $A$ be a $n \times n$ random matrix with i.i.d. entries $A_{ij} = \xi_{ij}\d_{ij}$ where
    \begin{enumerate}
        \item $\xi_{ij}$ and $\d_{ij}$ are all independent.
        \item $\xi_{ij}$ are i.i.d. sub-Gaussian random variables with 
            \[
                \bE\xi_{ij} = 0, \quad \bE\xi_{ij}^2 = 1, \quad \xi_{ij} \in \bR.
            \]
        \item $\d_{ij}$ are i.i.d. copies of $\d_p$.
    \end{enumerate}
    Then there exist constants $c, C, C' > 0$ depending on the sub-Gaussian norm of $\xi_{ij}$ such that for $p \ge C' \frac{\log n}{n}$ we have
    \[
        \bP[\s_n(A) \le \e] \le e^{-cnp} + C\e \exp\left( c \frac{\log{(1/p)}}{\log(np)}\right) \cdot \sqrt{\frac{n}{p}}, \quad \e \ge 0.
    \]
\end{thm}
Let us consider a special case of Theorem $\ref{ru}$ where we fix $\d \in (0, 1)$ and set 
\[
    p = n^{\d-1} = \frac{n^\d}{n}.
\]
Then there exist constants $c, C > 0$ depending on the sub-Gaussian norm of $\xi_{ij}$ such that for sufficiently large $n$ we have
\begin{align*}
    \bP[\s_n(A) \le \e] &\le e^{-cn^\d} + C\e \exp\left( c \frac{1 - \d}{\d}\right) \cdot n^{\frac{2 - \d}{2}}\\
                        &= e^{-cn^\d} + O(\e n^{\frac{2 - \d}{2}}).
\end{align*}
In this paper, we will show that the dependence on $\e$ can be upgraded to $\e^2$ if we move from real sub-Gaussian entries to complex Gaussian entries. More formally, our main result is:
\begin{thm}
    \label{main}
    Fix $\d \in (0, 1)$ and let $A$ be the $n \times n$ random matrix in Theorem \ref{ru} but with
    \[
        \xi_{ij} \sim \cN_\bC(0,1), \quad \d_{ij} \overset{d}{=} \d_{n^{\d-1}}.
    \]
    Then there exists a constant $c > 0$ such that for sufficiently large $n$ we have
    \[
        \bP[\s_n(A) \le \e] \le e^{-cn^\d} + O(\e^2n^{2-\d}), \quad \e \ge 0.
    \]
\end{thm}
Here and throughout this paper, $\cN_\bC(0,1)$ will denote the complex normal distribution with mean 0 and variance 1.

\section{Proof Overview}
From now on, $A$ will refer to the random matrix as defined Theorem \ref{main}.

The proof of Theorem \ref{main} will require careful analysis of nets on the unit sphere. Because such arguments need upper bound estimates on $\|A\|_{op}$, we will break up the proof of Theorem \ref{main} into two parts:
\begin{thm}
    \label{mainwnorm}
    For any $K > 0$, there exists a constant $c > 0$ such that for sufficiently large $n$ we have
    \[
        \bP[\s_n(A) \le \e \text{ and } \|A\|_{op} \le Kn^{\frac{\d}{2}}] \le e^{-cn^\d} + O(\e^2n^{2 - \d}), \quad 0 \le \e \le 1.
    \]
\end{thm}
\begin{thm}
    \label{opbound}
    There exists constants $K, c > 0$ such that for sufficiently large $n$ we have
    \[
        \bP[\|A\|_{op} \ge Kn^{\frac{\d}{2}}] \le e^{-cn^\d}.
    \]
\end{thm}
Combining Theorems \ref{mainwnorm} and \ref{opbound} with a union bound gives Theorem \ref{main}.
\begin{remark}
    Careful readers may have noticed that Theorem \ref{mainwnorm} restricts $\e$ to $[0, 1]$. It is very easy to extend the statement to all $\e \ge 0$; the added restriction will just make the proof a little simpler.
\end{remark}

\subsection{Theorem \ref{mainwnorm} Proof Overview}
As the proof of Theorem \ref{mainwnorm} comprises the bulk of this paper, we outline it here.

 Let $Y_i$ denote the $i$-th column of $A$ and let $W_i$ denote the span of the $Y_j$'s for $j \neq i$. Using standard arguments, we show it suffices to prove that it is unlikely for $\dist(Y_1, W_1)$ to be small. Observe that
\begin{align}
    \label{overview-dist-small-equiv}
    \dist(Y_1, W_1) \ \text{small} &\iff \left|\langle Y_1, \eta_1\rangle\right| \ \text{small for any} \ \eta_1 \in W_1^\perp, \ \|\eta_1\|_2 = 1.
\end{align}
Notice that $\eta_1 \in W_1^\perp \iff \eta_1 \in \ker(B)$, where $A = [Y_1 \ B^*]$. Writing $\cS^{n-1} \subset \bC^n$ as the unit sphere, we can rewrite (\ref{overview-dist-small-equiv}) as
\[
    \dist(Y_1, W_1) \ \text{small} \iff \ \left|\langle Y_1, \eta_1\rangle\right| \ \text{small for any} \ \eta_1 \in \ker(B) \cap \cS^{n-1}.
\]
We show that for generic vectors $\eta \in \cS^{n-1}$ we have good control on the lower tail of $\left|\langle Y_1, \eta\rangle\right|$. In the rare case that $\eta$ is structured and does not provide a good lower tail bound on $\left|\langle Y_1, \eta\rangle\right|$, we show that it is unlikely $\eta \in \ker(B)$. Combining these arguments will then provide us a proof of Theorem \ref{mainwnorm}.

This general strategy of reducing the least singular problem into a statement of distance between columns and then doing casework was developed by Rudelson in \cite{Ru08} and has since become the standard technique in attacking such problems. \cite{BaRu17} and this paper are no exceptions to the rule. However, the technical difficulties that arise when working with sparse matrices are markedly different than in other settings. We will draw some inspiration from \cite{BaRu17} because of this, but our casework and analysis will differ from \cite{BaRu17} significantly.

\subsection{Generic and Structured Vectors} \label{generic}
Let us now motivate our definitions of generic and structured vectors by considering the distribution of $|\langle Y_1, \eta \rangle|$ for $\eta \in \cS^{n-1}$.

Let $\hat{\eta}, G \in \bC^n$ denote the random vectors
\[
    \hat{\eta}_i = \d_{i1}\eta_i, \quad G_i = \xi_{i1}.
\]
(Notice that $G$ is a Gaussian random vector.) Then we can rewrite
\[
    |\langle Y_1, \eta \rangle| = |\langle G, \hat{\eta} \rangle|.
\]
Observe if we condition on the value of $\|\hat{\eta}\|_2$, we obtain
\[
    \langle G, \hat{\eta} \rangle \sim \cN_\bC(0, \|\hat{\eta}\|_2).
\]
Consequently, if we want a good lower tail bound on $|\langle Y_1, \eta \rangle|$, we want it to be unlikely for $\|\hat{\eta}\|_2$ to be small.

How do we make it unlikely for $\|\hat{\eta}\|_2$ to be small? One way is to require $\eta$ to have many coordinates of a certain size so that it is likely some coordinates of this size will not be zeroed out. This leads us to our definition.

Partition the unit sphere $\cS^{n-1}$ into the sets
\begin{align*}
    \cS_{IC} &= \left\{x \in \cS^{n-1} : \sum_{|x_i|^2 \le \frac{1}{c_1n}} |x_i|^2 \ge \e_1\right\}\\
    \cS_{MC} &= \left\{x \in \cS^{n-1} : \sum_{|x_i|^2 \le \frac{1}{c_1n}} |x_i|^2 < \e_1, \quad \sum_{\frac{1}{c_1n} < |x_i|^2 \le \frac{1}{c_2n^{1-\d}}} |x_i|^2 \ge \e_2\right\}\\
    \cS_{HC} &= \left\{x \in \cS^{n-1} : \sum_{|x_i|^2 \le \frac{1}{c_2n^{1-\d}}} |x_i|^2 < \e_1 + \e_2\right\}
\end{align*}
where $c_1, c_2, \e_1, \e_2$ are constants that will be determined later. We call $\cS_{IC}$ the set of incompressible (or generic) vectors and call $\cS^{n-1} - \cS_{IC}$ the set of compressible (or structured) vectors. In particular, we will call $\cS_{MC}$ the set of moderately compressible vectors, and $\cS_{HC}$ the set of highly compressible vectors. In the following sections, we first show that when $\eta$ is compressible, it is likely that $\eta \not \in \ker(B)$, or equivalently,
\[
    \|B\eta\|_2 > 0.
\]
For presentation purposes (and to avoid proof duplication), we will do this by showing it is likely that
\[
    \|A\eta\|_2
\]
cannot be too small for compressible $\eta$. The argument will work in exactly the same way when $A$ is replaced by $B$. Because of technical issues introduced by sparsity, we will treat the case $\eta \in \cS_{HC}$ (Section \ref{hc}) differently from the case $\eta \in \cS_{MC}$ (Section \ref{mc}).

On the other hand, in Section \ref{ic} we show when $\eta$ is incompressible, it is likely that
\[
    \left|\langle Y_1, \eta\rangle\right|
\]
cannot be too small. We then put together our arguments in Section \ref{mainwnormproof} to prove Theorem \ref{mainwnorm}.

\section{Notation}
For the reader's convenience, we collect some notation that will continue to appear throughout the paper.
\begin{center}
    \begin{tabular}{|l|l|}
        \hline
        \textbf{Symbol} & \textbf{Meaning} \\
        \hline
        $\d_p$                          & As defined in Section \ref{intro}.\\
        $n, A, \d_{ij}, \xi_{ij}, \d$   & As defined in Theorem \ref{main}.\\
        $Y_i$                           & $i$-th column of $A$. \\
        $W_i$                           & Span of the $Y_j$'s for $j \neq i$. \\
        $B$                             & Matrix such that $A = [Y_1 \ B^*]$.\\
        $\cS^{n-1}$                     & Complex unit sphere $\cS^{n-1} \subset \bC^n$.\\
        $\cS_{\{HC, MC, IC\}}$          & As defined in Section \ref{generic}.\\
        $c_1, c_2, \e_1, \e_2$          & As defined in Section \ref{generic}.\\
        $\cN_\bC(\mu,\s^2)$             & Complex Gaussian distribution.\\
        $\cN_\bR(\mu,\s^2)$             & Real Gaussian distribution. \\
        \hline
    \end{tabular}
\end{center}

\section{Preliminaries}
In this section we will investigate the properties of the set
\[
    V = \left\{x \in \cS^{n-1} : \sum_{|x_i|^2 \le \frac{1}{a}} |x_i|^2 < d_1, \quad \sum_{\frac{1}{a} < |x_i|^2 \le \frac{1}{b}} |x_i|^2 \ge d_2 \right\}.
\]
We will assume $0 < d_1 \le 0.1$ and $d_1 \ll d_2$, which hold in all our use cases.

Define
\[
    \Sparse(m) = \{x \in \bC^n : |\supp(x)| \le m\}.
\]
Let $\cN$ be a maximal $\sqrt{d_1}$-net of the $\Sparse(a)$ vectors on $\cS^{n-1}$. Let us approximate any $x \in V$ with an element in $\cN$ in the following manner:
\begin{enumerate}
    \item Set $x^{(0)} = x \in V$.
    \item Define $x^{(1)} \in \bC^n$ such that
        \[
            x^{(1)}_i =
            \begin{cases}
                x^{(0)}_i & \frac{1}{a} < |x^{(0)}_i|^2 \\
                0 & \text{otherwise}
            \end{cases}.
        \]
    \item Set $x^{(2)} = x^{(1)} / \|x^{(1)}\|_2$.
    \item Let $x^{(3)} \in \cN$ be $\sqrt{d_1}$-approximation of $x^{(2)}$.
\end{enumerate}
Notice that
\[
    \|x - x^{(3)}\|_2 \le \sum_{i = 0}^2 \|x^{(i)} - x^{(i+1)}\|_2 \le 3\sqrt{d_1}.
\]
so every point in $V$ is $3\sqrt{d_1}$-close to a point in $\cN$.

However, we are not satisfied with $x^{(3)}$ just being a $\Sparse(a)$ approximation of $x$; we would like $x^{(3)}$ to inherit some properties of $x$ as well. In fact, we claim
\begin{equation}
    \label{x3claim}
    \sum_{\frac{1}{2a} < |x_i^{(3)}|^2 \le \frac{4}{b}} |x_i^{(3)}|^2 \ge d_2 - 57\sqrt{d_1}.
\end{equation}
To see this, first note
\begin{align*}
    \sum_{\frac{1}{a} < |x_i|^2 \le \frac{1}{b}} |x_i|^2 \ge d_2 &\implies \sum_{\frac{1}{a} < |x^{(1)}_i|^2 \le \frac{1}{b}} |x^{(1)}_i|^2 \ge d_2 \\
                                                                           &\implies \sum_{\frac{1}{a} < |x^{(2)}_i|^2 \le \frac{1}{b(1 - \sqrt{d_1})^2}} |x^{(2)}_i|^2 \ge d_2 \\
                                                                           &\implies \sum_{\frac{1}{a} < |x^{(2)}_i|^2 \le \frac{3}{b}} |x^{(2)}_i|^2 \ge d_2.
\end{align*}
Define $LB, LA, NL \subset \{1, \dots, n\}$ to be the sets of indices
\begin{align*}
    i \in LB &\implies \frac{1}{a} < |x^{(2)}_i|^2 \le \frac{3}{b}, \quad |x^{(3)}_i|^2 \le \frac{1}{2a} \\
    i \in LA &\implies \frac{1}{a} < |x^{(2)}_i|^2 \le \frac{3}{b}, \quad |x^{(3)}_i|^2 > \frac{4}{b} \\
    i \in NL &\implies \frac{1}{a} < |x^{(2)}_i|^2 \le \frac{3}{b}, \quad \frac{1}{2a} < |x^{(3)}_i|^2 \le \frac{4}{b}.
\end{align*}
(Think of $LB$ standing for ``loss from below, " $LA$ standing for ``loss from above," and $NL$ standing for ``no loss.") Observe we have
\begin{align*}
    d_1 &\ge \sum_{i \in LB} \left(|x_i^{(2)}| - |x_i^{(3)}|\right)^2 \\
        &\ge \sum_{i \in LB} \left(|x_i^{(2)}| - \frac{|x_i^{(2)}|}{\sqrt{2}}\right)^2 \\
        &\ge \left(1 - \frac{1}{\sqrt{2}}\right)^2 \sum_{i \in LB} |x_i^{(2)}|^2
\end{align*}
and
\begin{align*}
    d_1 &\ge \sum_{i \in LA} \left(|x_i^{(3)}| - |x_i^{(2)}|\right)^2 \\
        &\ge \sum_{i \in LA} \left(\sqrt{\frac{4}{b}} - \sqrt{\frac{3}{b}}\right)^2 \\
        &\ge \left(\sqrt{\frac{4}{3}} - 1\right)^2 \sum_{i \in LA} \frac{3}{b} \\
        &\ge \left(\sqrt{\frac{4}{3}} - 1\right)^2 \sum_{i \in LA} |x^{(2)}_i|^2.
\end{align*}
Thus we obtain
\begin{equation}
    \label{LB-and-LA}
    \sum_{i \in LB \cup LA} |x_i^{(2)}|^2 \le 12d_1 + 42d_1 = 54d_1.
\end{equation}
Furthermore, we have
\begin{align}
    \sum_{i \in NL} |x^{(2)}_i|^2 - \sum_{i \in NL} |x_i^{(3)}|^2 &= \sum_{i \in NL} |x^{(2)}_i|^2 - |x_i^{(3)}|^2 \nonumber \\
                                                                  &\le \sum_{i \in NL} |x^{(2)}_i - x^{(3)}_i|^2 + 2|x^{(3)}_i| \cdot |x^{(2)}_i - x^{(3)}_i| \nonumber \\
                                                                  &\le \|x^{(2)} - x^{(3)}\|^2_2 + 2 \| x^{(3)} \|_2 \|x^{(2)} - x^{(3)}\|_2 \nonumber \\
                                                                  &\le d_1 + 2\sqrt{d_1}. \label{NL}
\end{align}
From (\ref{LB-and-LA}) and (\ref{NL}) we get
\[
    \sum_{\frac{1}{a} < |x^{(2)}_i|^2 \le \frac{3}{b}} |x^{(2)}_i|^2 - \sum_{\frac{1}{2a} < |x_i^{(3)}|^2 \le \frac{4}{b}} |x_i^{(3)}|^2 \le 54d_1 + d_1 + 2\sqrt{d_1} \le 57\sqrt{d_1},
\]
proving (\ref{x3claim}).

Consequently, every point in $V$ is $3\sqrt{d_1}$-close to a point in
\[
    \Sigma_V := \cN \cap \left\{x \in \cS^{n-1} : \sum_{\frac{1}{2a} < |x_i|^2 \le \frac{4}{b}} |x_i|^2 \ge d_2 - 57\sqrt{d_1}\right\}.
\]
It is also well-known that
\[
    |\cN| < \left(\frac{3}{\sqrt{d_1}}\right)^{2a} \binom{n}{a} < \left(\frac{3}{\sqrt{d_1}}\right)^{2a} \left(\frac{ne}{a}\right)^a.
\]
(For a proof, see Lemma 5.2 in \cite{Ver11}). Hence we have shown
\begin{lem}
    \label{net}
    Let $V$ be defined as above with $0 < 57\sqrt{d_1} < d_2 < 1$. Then every point in $V$ is $3\sqrt{d_1}$-close to a point in $\Sigma_V$, where
    \[
        |\Sigma_V| < \left(\frac{3}{\sqrt{d_1}}\right)^{2a} \left(\frac{ne}{a}\right)^a
    \]
    and
    \[
        x \in \Sigma_V \implies 
        \begin{cases}
            x \in \Sparse(a) &\\
            \sum_{\frac{1}{2a} < |x_i|^2 \le \frac{4}{b}} |x_i|^2 \ge d_2 - 57\sqrt{d_1} &
        \end{cases}.
    \]
\end{lem}
Sometimes, instead of studying $\|Ax\|$ for $x \in \Sigma_V$ directly, we would like to first zero out some entries of $x$. Hence we will need:
\begin{lem} \label{zero-out}
    Suppose $x, y \in \bC^n$ such that $|x_i| \ge |y_i|$ for all $1 \le i \le n$.
    Then
    \[
        \bP\left[\|Ax\|_2 \le t\right] \le \bP\left[\|Ay\|_2 \le t\right], \quad t \ge 0.
    \]
\end{lem}
\begin{proof}
    The key is to exploit the fact that the nonzero entries of $A$ are drawn from the Gaussian distribution. Note we can rewrite
    \begin{align*}
        \bP\left[\|Ax\|_2 \le t\right] &= \bP\left[\|Ax\|^2_2 \le t^2\right] \\
                                       &= \bP\left[\sum_{i=1}^n |R_i \cdot x|^2 \le t^2\right],
    \end{align*}
    where $\sum |R_i \cdot x|^2$ is a sum of independent random variables.

    Let us first consider $|R_1 \cdot x|$. If we condition on the values of $\d_{11}, \dots, \d_{1n}$, then
    \[
        |R_1 \cdot x| \sim \cN_{\bC}\left(0, \sum_{i = 1}^n \d_{1i}|x_i|^2\right), \quad |R_1 \cdot y| \sim \cN_{\bC}\left(0, \sum_{i = 1}^n \d_{1i}|y_i|^2\right).
    \]
    As
    \[
        \sum_{i = 1}^n \d_{1i}|x_i|^2 \ge \sum_{i = 1}^n \d_{1i}|y_i|^2,
    \]
    we conclude
    \[
        \bP\left[|R_1 \cdot x| \le t' \ \Big| \ \d_{1i}, \dots, \d_{1n} \right] \le \bP\left[|R_1 \cdot y| \le t' \ \Big| \ \d_{1i}, \dots, \d_{1n} \right], \quad t' \ge 0,
    \]
    which implies
    \[
        \bP\left[|R_1 \cdot x| \le t'\right] \le \bP\left[|R_1 \cdot y| \le t' \right], \quad t' \ge 0.
    \]
    Hence
    \[
        \bP\left[\sum_{i=1}^n |R_i \cdot x|^2 \le t^2\right] \le \bP\left[|R_1 \cdot y|^2 + \sum_{i=2}^n |R_i \cdot x|^2 \le t^2\right].
    \]
    The lemma follows by repeating this argument for $i = 2, \dots, n$.
\end{proof}
We are now ready to tackle Theorem \ref{mainwnorm}.

\section{The Highly Compressible Case} \label{hc}
In this section, we would like to show
\begin{thm}
    \label{hsparse}
    There exists a constant $c > 0$ such that for sufficiently large $n$ we have
    \[
        \bP\left[\inf_{x \in \cS_{HC}} \|Ax\|_2 \le \e \text{ and } \|A\|_{op} \le Kn^{\frac{\d}{2}}\right] \le e^{-cn^\d}, \quad 0 \le \e \le 1.
    \]
\end{thm}
Let $R_i$ denote the $i$-th row of $A$. When $x$ is highly compressible, most of the mass of $x$ is concentrated on $O(n^{1 - \d})$ coordinates. Consequently, as each entry of $A$ is equal to 0 with probability $n^{\d - 1}$, it is difficult to obtain a good lower tail bound on
\[
    \|Ax\|_2^2 = \sum_i |R_i \cdot x|^2
\]
using standard arguments like Paley-Zygmund or Chernoff bounds. Instead, we employ an alternative strategy by exploiting the fact that there is very little cancellation in the sum
\[
    R_i \cdot x = \sum_j A_{ij}x_j.
\]
To do this, we introduce the following lemma.
\begin{lem}[Row Bound]
    \label{rowbound}
    For $y \in \bC^n$ and $J \subset \{1, \dots, n\}$, define $I_y(J)$ to be the set
    \[
        \left\{i \in [1,n] : |A_{i,j^*}| \ge 1 \text{ for some $j^* \in J$, } A_{i, j} = 0 \text{ for $j \in \supp(y)\backslash j^*$}\right\}.
    \]
    Then
    \[
        \bP\left[|I_y(J)| \le \frac{3}{20}|J|n^\d\left(1 - n^{\d-1}\right)^{m-1}\right] \le \exp\left(-\frac{3}{80}|J|n^\d\left(1 - n^{\d-1}\right)^{m-1}\right),
    \]
    where $m = |\supp(y)|$.
\end{lem}
\begin{proof}
    Denote
    \[
        c_g = \bP[|X| \ge 1], \quad X \sim \cN_\bC(0,1)
    \]
    and observe
    \[
        \bP[i \in I_y(J)] = c_g|J|n^{\d-1}\left(1 - n^{\d-1}\right)^{m-1} > \frac{3}{10}|J|n^{\d-1}\left(1 - n^{\d-1}\right)^{m-1}.
    \]
    The result then follows by applying the multiplicative Chernoff bound.
\end{proof}
\begin{remark}
    Notice that Lemma \ref{rowbound} is strong for small $m$, i.e. when $y$ is sparse.
\end{remark}
As an illustration of our argument, let us suppose $x \in \cS^{n-1}$ is a $O(n^{1-\d})$-sparse vector. Without loss of generality suppose the entries of $x$ are in non-increasing order of magnitude. Let us denote
\[
    x_{[a:b]} = [0, \dots, 0, x_a, x_{a+1}, \dots, x_b, 0, \dots, 0]^T.
\]
Then for any step size $s$, we note
\[
    x = \sum_{i=0}^{\frac{n}{s} - 1} x_{[is + 1: (i + 1)s]}.
\]
Consequently, we have
\[
    \|Ax\|^2_2 \ge \sum_{i=0}^{\frac{n}{s} - 1} |I_x\left([is + 1, (i + 1)s]\right)||x_{(i+1)s}|^2.
\]
The sparsity of $x$ implies that
\[
    \left(1 - n^{\d-1}\right)^{|\supp(x)| - 1}
\]
is bounded below by a constant, so from Lemma \ref{rowbound} followed by a union bound we conclude that with probability
\[
    \left(\frac{n}{s}\right)e^{-O(n^\d s)}
\]
we have
\[
    \|Ax\|_2^2 \le Cn^\d\sum_{i = 0}^{\frac{n}{s} - 1} s|x_{(i+1)s}|^2
\]
for some constant $C > 0$. Pictorally, the sum $\sum_{i = 0}^{\frac{n}{s} - 1} s|x_{(i+1)s}|^2$ can be drawn as follows:
\begin{figure}[H]
    \center
    \def\svgwidth{0.5\columnwidth}
    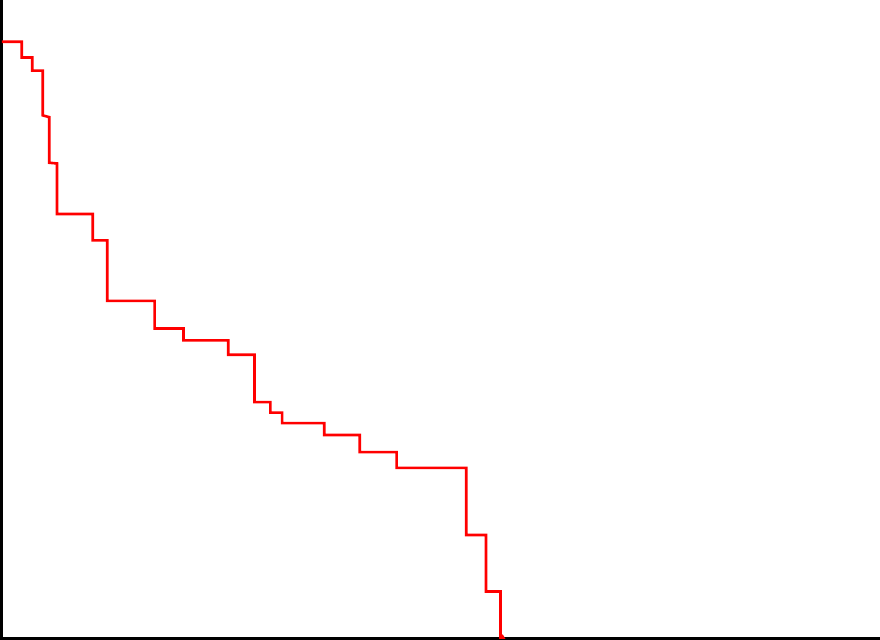
\end{figure}
Notice this sum is maximized when the step size $s = 1$; however, the resulting probability
\[
    ne^{-O(n^\d)}
\]
is not good enough to overcome the entropy of the net of $O(n^{1-\d})$-sparse vectors. Consequently, we will disjointize $\cS_{HC}$ into many subsets, using smaller values of $s$ when we can afford to and using larger values when we cannot. In the following subsections, we will make this argument rigourous.

\begin{remark}
    In our illustration above, we have implicitly assumed that $n$ is divisible by the step size $s$, and so $\frac{n}{s}$ is an integer. This need not be the case; if $n$ is not divisible by $s$, we will have to consider $\lfloor \frac{n}{s} \rfloor$ instead (along will all the necessary modifications). The argument, however, will be exactly the same. Consequently, for notational convenience, we will carry out all out computations as if $n$ is divisible by $s$ --- but this need not be the case.
\end{remark}

\subsection{}
In this subsection, we will bound
\[
    \bP\left[\inf_{x \in V} \|Ax\|_2 \le \e \text{ and } \|A\|_{op} \le Kn^{\frac{\d}{2}}\right], \quad 0 \le \e \le 1.
\]
where
\[
    V = \left\{x \in \cS^{n-1} : \sum_{|x_i|^2 \le \frac{1}{a}} |x_i|^2 < d_1, \quad \sum_{\frac{1}{a} < |x_i|^2 \le \frac{1}{b}} |x_i|^2 \ge d_2 \right\}.
\]
with $a \le c_2n^{1 -\d}$ and $d_2/\sqrt{d_1}$ sufficiently large.

From Lemma \ref{net}, we have
\begin{align*}
    &\bP\left[\inf_{x \in V} \|Ax\|_2 \le \e \text{ and } \|A\|_{op} \le Kn^{\frac{\d}{2}}\right] \\
    &\le \bP\left[\inf_{x \in \Sigma_V}\|Ax\|_2 \le \e + 3K\sqrt{d_1}n^{\frac{\d}{2}} \text{ and } \|A\|_{op} \le Kn^{\frac{\d}{2}}\right] \\
    &\le \sum_{x \in \Sigma_V} \bP\left[\|Ax\|_2 \le \e + 3K\sqrt{d_1}n^{\frac{\d}{2}} \text{ and } \|A\|_{op} \le Kn^{\frac{\d}{2}}\right] \\
    &\le \sum_{x \in \Sigma_V} \bP\left[\|Ax\|_2 \le \e + 3K\sqrt{d_1}n^{\frac{\d}{2}} \right],
\end{align*}
where
\[
    x \in \Sigma_V \implies \sum_{\frac{1}{2a} < |x_i|^2 \le \frac{4}{b}} |x_i|^2 \ge d_2 - 57\sqrt{d_1}.
\]
As $\e$ is bounded, for sufficiently large $n$ we have
\[
    \bP\left[\inf_{x \in V} \|Ax\|_2 \le \e \text{ and } \|A\|_{op} \le Kn^{\frac{\d}{2}}\right] \le \sum_{x \in \Sigma_V} \bP\left[\|Ax\|_2 \le 4K\sqrt{d_1}n^{\frac{\d}{2}} \right].
\]
Fix an arbitrary $x \in \Sigma_V$ and let $y$ denote the vector $x$ but with the entries $|x_i|^2 > \frac{4}{b}$ set to 0. By Lemma \ref{zero-out}, we have
\[
    \bP\left[\|Ax\|_2 \le 4K\sqrt{d_1}n^{\frac{\d}{2}} \right] \le \bP\left[\|Ay\|_2 \le 4K\sqrt{d_1}n^{\frac{\d}{2}} \right]
\]
so it suffices to bound the right hand side.

Without loss of generality assume the entries of $y$ are in non-increasing order of magnitude. For sufficiently large $n$, observe Lemma \ref{rowbound} implies
\begin{align*}
    \bP\left[|I_y(J)| \le \frac{1}{10e^{c_2}}|J|n^\d\right] \le \exp\left(-\frac{1}{40e^{c_2}}|J|n^\d\right).
\end{align*}
Then with probability
\[
    \ge 1 - \frac{n}{s} \exp\left(-\frac{1}{40e^{c_2}}sn^\d\right)
\]
we have
\begin{align*}
    \|Ay\|_2^2 &\ge \frac{1}{10e^{c_2}}sn^\d\sum_{i=0}^{\frac{n}{s} - 1} |y_{(i+1)s}|^2 \\
               &\ge \frac{1}{10e^{c_2}}n^\d\|y_{[s:n]}\|_2^2 \\
               &\ge \frac{1}{10e^{c_2}}n^\d\left(d_2 - 57\sqrt{d_1} - \frac{4(s - 1)}{b}\right).
\end{align*}
Setting
\[
    s = 1 + \lfloor \frac{(d_2 - 57\sqrt{d_1})b}{8}\rfloor
\]
we obtain
\[
    \bP\left[\|Ay\|_2^2 \le \frac{1}{20e^{c_2}}n^\d(d_2 - 57\sqrt{d_1})\right] \le n \exp\left(-\frac{1}{40e^{c_2}}\left( 1 + \lfloor\frac{(d_2 - 57\sqrt{d_1})b}{8}\rfloor\right) n^\d\right)
\]
for sufficiently large $n$. Thus we get
\begin{lem}
    \label{hcstep}
    Let $V$ be defined as before with $a \le c_2n^{1 -\d}$ and 
    \[
        0 < 57\sqrt{d_1} + 320e^{c_2}K^2d_1 < d_2 < 1.
    \]
    Then for $0 \le \e \le 1$ and sufficiently large $n$ we have
    \begin{align*}
        &\bP\left[\inf_{x \in V} \|Ax\|_2 \le \e \text{ and } \|A\|_{op} \le Kn^{\frac{\d}{2}}\right] \\
        &\le n\exp\left(-\frac{1}{40e^{c_2}}\left(1 + \lfloor \frac{(d_2 - 57\sqrt{d_1})b}{8}\rfloor\right) n^\d\right) \left(\frac{3}{\sqrt{d_1}}\right)^{2a} \left(\frac{ne}{a}\right)^a.
    \end{align*}
\end{lem}
Lemma \ref{hcstep} will allow us to prove Theorem \ref{hsparse}.

\subsection{Proof of Theorem \ref{hsparse}}
\begin{proof}
    Let $m$ be the largest integer such that
    \[
        (m - 1)\frac{\d}{2} < 1 - \d.
    \]
    Write $S_{HC}$ as the finite union
    \[
        \cS_{HC} = \bigcup_{k=1}^{m} \cS_{HC}^k,
    \]
    where
    \[
        \cS_{HC}^k = \left\{x \in \cS^{n-1} : \sum_{|x_i|^2 \le \frac{1}{a_k}} |x_i|^2 < d_{1k}, \quad \sum_{\frac{1}{a_k} < |x_i|^2 \le \frac{1}{b_k}} |x_i|^2 \ge d_{2k} \right\}
    \]
    and $a_k, b_k, d_{1k}, d_{2k} > 0$ are finite sequences such that
    \begin{align*}
        &a_k = n^{\frac{k\d}{2}} \text{ for } k < m, \quad a_m = c_2n^{1-\d}, \quad b_k = a_{k - 1}, \quad b_1 = 1, \\
        &d_{21} + d_{11} = 1, \quad d_{1m} = \e_1 + \e_2, \quad d_{1k} = d_{1(k+1)} + d_{2(k+1)}, \\
        &d_{2k} > 57\sqrt{d_{1k}} + 320e^{c_2}K^2d_{1k}.
    \end{align*}
    (Note the $d_{1k}, d_{2k}$'s will only exist if $\e_1 + \e_2$ is very small, which we can enforce with our choice of $\e_1$ and $\e_2$ in Section \ref{mc}.) Then by Lemma \ref{hcstep}, we get
\begin{align*}
    &\bP\left[\inf_{x \in \cS_{HC}} \|Ax\|_2 \le \e \text{ and } \|A\|_{op} \le Kn^{\frac{\d}{2}}\right] \\
    &\le \sum_{k} \bP\left[\inf_{x \in \cS_{HC}^k} \|Ax\|_2 \le \e \text{ and } \|A\|_{op} \le Kn^{\frac{\d}{2}}\right] \\
    &\le e^{-cn^\d}
\end{align*}
for some $c > 0$ and sufficiently large $n$, as desired.
\end{proof}

\section{The Moderately Compressible Case} \label{mc}
In this section, we would like to show
\begin{thm}
    \label{msparse}
    There exists a constant $c > 0$ such that for sufficiently large $n$ we have
    \[
        \bP\left[\inf_{x \in \cS_{MC}} \|Ax\|_2 \le \e \text{ and } \|A\|_{op} \le Kn^{\frac{\d}{2}}\right] \le e^{-cn}, \quad 0 \le \e \le 1.
    \]
\end{thm}
Moderately compressible vectors are generic enough for us to use a more traditional argument.

From Lemma \ref{net} there exists a net $\Sigma_{MC} := \Sigma_{S_{MC}}$ such that every point in $S_{MC}$ is $3\sqrt{\e_1}$-close to a point in $\Sigma_{MC}$ and
\[
    x \in \Sigma_{MC} \implies \sum_{\frac{1}{2c_1n} < |x_i|^2 \le \frac{4}{c_2n^{1-\d}}} |x_i|^2 \ge \e_2 - 57\sqrt{\e_1}.
\]
Thus we have
\begin{align*}
    &\bP\left[\inf_{x \in \cS_{MC}}\|Ax\|_2 \le \e \text{ and } \|A\| \le Kn^{\frac{\d}{2}}\right] \\
    &\le \sum_{x \in \Sigma_{MC}} \bP\left[\|Ax\|_2 \le 4K\sqrt{\e_1}n^{\frac{\d}{2}} \right] \\
    &= \sum_{x \in \Sigma_{MC}} \bP\left[\sum^{n}_{i=1}|R_i \cdot x|^2 \le (4K\sqrt{\e_1}n^{\frac{\d}{2}})^2 \right], \quad R_i \ i\text{-th row} \\
    &\le \sum_{x \in \Sigma_{MC}} \bP\left[|R_i \cdot x|^2 \le \frac{(4K\sqrt{\e_1})^2}{n^{(1-\d)}\d'} \text {for } (1-\d')n \text{ values of } i\right], \quad \d' > 0 \\
    &\le \sum_{x \in \Sigma_{MC}} \binom{n}{(1-\d')n} \bP\left[|R_1 \cdot x|^2 \le \frac{(4K\sqrt{\e_1})^2}{n^{(1-\d)}\d'} \right]^{(1-\d')n}, \quad \d' > 0 \\
    &= \sum_{x \in \Sigma_{MC}} \binom{n}{\d' n} \bP\left[|R_1 \cdot x| \le \frac{4K\sqrt{\e_1}}{n^{\frac{1-\d}{2}}\sqrt{\d'}} \right]^{(1-\d')n}, \quad \d' > 0.
\end{align*}
To bound this probability, we will need a few intermediate results.
\begin{lem}
    \label{cgaussian}
    Let $X \sim \cN_\bC(0, \s^2)$. Then
    \[
        \bP[|X| \le \e] \le \frac{\e^2}{\s^2}.
    \]
\end{lem}
\begin{proof}
    Let $X_1 = \Re(X)$ and $X_2 = \Im(X)$. Note $X_1, X_2 \sim \cN_\bR(0, \frac{\s^2}{2})$ and are independent. Furthermore, we have
    \[
        \frac{2}{\s^2}|X|^2 = \frac{2}{\s^2}X_1^2 + \frac{2}{\s^2}X_2^2,
    \]
    so $\frac{2}{\s^2}|X|^2$ is the $\chi^2$-distribution with 2 degrees of freedom. Consequently, we have
    \[
        \begin{split}
            \bP\left[ \frac{2}{\s^2}|X|^2 \le t \right] &= 1 - e^{-\frac{t}{2}} \\
                                                      &= e^{-\frac{t}{2}}(e^{\frac{t}{2}} - 1) \\
                                                      &= e^{-\frac{t}{2}}\left( \sum_{i=1}^\i \left( \frac{t}{2}\right)^i \frac{1}{i!} \right) \\
                                                      &= e^{-\frac{t}{2}}\frac{t}{2}\left( \sum_{i=1}^\i \frac{1}{i}\left( \frac{t}{2}\right)^{i-1} \frac{1}{(i-1)!} \right) \\
                                                      &\le e^{-\frac{t}{2}}\frac{t}{2}e^{\frac{t}{2}} \\
                                                      &= \frac{t}{2}.
        \end{split}
    \]
    Thus we obtain
    \[
        \bP[|X| \le \e] = \bP[|X|^2 \le \e^2] = \bP\left[ \frac{2}{\s^2}|X|^2 \le \frac{2\e^2}{\s^2}\right] \le \frac{\e^2}{\s^2}.
    \]
\end{proof}

\begin{lem}
    \label{sbound}
    Suppose
    \[
        \sum_{i = 1}^n a_i = a, \quad a_i \ge 0.
    \]
    Let the $b_i$'s be independent random variables such that $b_i \overset{d}{=} a_i\d_{n^{\d-1}}$. Then
    \[
        \bP\left[ \sum_{i = 1}^n b_i \le n^{\d-1}ta \right] \le 1 - (1 - t)^2 \frac{a}{a + (n^{1-\d} - 1)\max a_i}
    \]
    for $0 \le t \le 1$.
\end{lem}

\begin{proof}
    From the Paley-Zygmund inequality, we get
    \begin{align*}
        \bP\left[ \sum_{i = 1}^n b_i \le n^{\d-1}ta \right] &\le 1 - (1 - t)^2 \frac{n^{2(\d-1)}a^2}{n^{2(\d-1)}a^2 + \Var\left(\sum b_j\right)} \\
                                                            &\le 1 - (1 - t)^2 \frac{n^{2(\d-1)}a^2}{n^{2(\d-1)}a^2 + \sum_j \left(n^{\d-1} a_j^2 - n^{2(\d-1)}a_j^2\right)} \\
                                                            &\le 1 - (1 - t)^2 \frac{a^2}{a^2 + (n^{1-\d} - 1) \sum_j a_j^2}
    \end{align*}
    from which the result follows.
\end{proof}
\begin{lem}
    \label{dotcontrol}
    For any $x \in \bC^n$, we have
    \[
        \bP\left[|R_1 \cdot x| \le \e\right] \le 1 - (1 - t)^2 \frac{\|x\|_2^2}{\|x\|_2^2 + (n^{1-\d} - 1)\sup|x_i|^2} + \frac{\e^2}{n^{\d-1}t\|x\|_2^2}
    \]
    for any $0 < t < 1$.
\end{lem}

\begin{proof}
    Define the random variable $\s^2 = \sum_{i: R_{1i} \neq 0} |x_i|^2$. By Lemma \ref{sbound}, we note that
    \[
        \bP\left[\s^2 \le n^{\d-1}t\|x\|_2^2\right] \le 1 - (1 - t)^2 \frac{\|x\|_2^2}{\|x\|_2^2 + (n^{1-\d} - 1)\sup|x_i|^2}. 
    \]
    Now assume that $\s^2 > n^{\d-1}t\|x\|^2_2$. Conditioning on the value of $\s^2$, we note that $R_1 \cdot x \sim \cN_\bC(0, \s^2)$. Thus Lemma \ref{cgaussian} implies
    \[
        \bP[|R_1 \cdot x| \le \e] \le \frac{\e^2}{\s^2} \le \frac{\e^2}{n^{\d-1}t\|x\|^2_2}
    \]
    Lemma \ref{dotcontrol} then follows from the union bound.
\end{proof}

Let us fix an arbitrary $x \in \Sigma_{MC}$ and let $y$ denote the vector $x$ with all entries $|x_i|^2 \ge \frac{4}{c_2n^{1-\d}}$ set to 0. Applying Lemma \ref{zero-out} and Lemma \ref{dotcontrol} then allows us to conclude that
\begin{align*}
    \bP\left[|R_1 \cdot x| \le \frac{4K\sqrt{\e_1}}{n^{\frac{1-\d}{2}}\sqrt{\d'}}\right] &\le \bP\left[|R_1 \cdot y| \le \frac{4K\sqrt{\e_1}}{n^{\frac{1-\d}{2}}\sqrt{\d'}}\right] \\
                                                                                         &\le 1 - (1 - t)^2 \frac{\|y\|_2^2}{1 + (n^{1-\d} - 1)\sup|y_i|^2} + \frac{16K^2\e_1}{t\|y\|_2^2\d'} \\
                                                                                         &\le 1 - (1 - t)^2 \frac{\e_2 - 57\sqrt{\e_1}}{1 + \frac{4}{c_2} - \frac{4}{c_2n^{1-\d}}} + \frac{16K^2\e_1}{t(\e_2 - 57\sqrt{\e_1})\d'}.
\end{align*}
As
\[
    |\Sigma_{MC}| \le \left(\frac{3}{\sqrt{\e_1}}\right)^{2c_1n}\left(\frac{e}{c_1}\right)^{c_1n} = \left(\frac{3}{\sqrt{\e_1}}\right)^{2c_1n}\left(\exp(c_1 (1 - \ln c_1))\right)^n
\]
and
\[
    \binom{n}{\d' n} \le \left(\frac{e}{\d'}\right)^{\d' n} = \left(\exp(\d' (1 - \ln \d'))\right)^n,
\]
to prove Theorem \ref{msparse} it suffices to choose appropriate values for $t, \e_1, \e_2, c_1, c_2$ and $\d'$ such that
\[
    \begin{split}
    &\left(\frac{3}{\sqrt{\e_1}}\right)^{2c_1n}\left(\exp(c_1 (1 - \ln c_1))\right)^n\left(\exp(\d' (1 - \ln \d'))\right)^n\\
    &\cdot \left(1 - (1 - t)^2 \frac{\e_2 - 57\sqrt{\e_1}}{1 + \frac{4}{c_2} - \frac{4}{c_2n^{1-\d}}} + \frac{16K^2\e_1}{t(\e_2 - 57\sqrt{\e_1})\d'}\right)^{(1-\d')n} \le e^{-cn}
    \end{split}
\]
for some $c > 0$ and sufficiently large $n$.

We will make our choices in the following (chronological) manner:
\begin{enumerate}
    \item Pick $c_2 > 0$ and pick $t \in (0, 1)$ such that $(1 - t)^2 = 0.5$.
    \item Pick $\e_2 > 0$ small enough such that if $\e_1$ is chosen to be $< \e_2$, then $\e_1 + \e_2$ is small enough for the proof of Theorem \ref{hsparse} to be valid.
    \item Pick $\d' > 0$ small enough such that
        \[
            \exp\left(\frac{\d'(1-\ln\d')}{1-\d'}\right) \cdot \left(1 - 0.4 \frac{\e_2}{1 + \frac{4}{c_2} - \frac{4}{c_2n^{1-\d}}}\right) < 1 - 0.3 \frac{\e_2}{1 + \frac{4}{c_2} - \frac{4}{c_2n^{1-\d}}}.
        \]
        Note this is possible as $\lim_{\d' \to 0} \d'(1 - \ln\d') = 0$.
    \item Pick $\e_1 > 0$ small enough such that
        \[
            1 - 0.5 \frac{\e_2 - 57\sqrt{\e_1}}{1 + \frac{4}{c_2} - \frac{4}{c_2n^{1-\d}}} + \frac{16K^2\e_1}{t(\e_2 - 57\sqrt{\e_1})\d'} < 1 - 0.4 \frac{\e_2}{1 + \frac{4}{c_2} - \frac{4}{c_2n^{1-\d}}}.
        \]
    \item Finally, pick $c_1 > 0$ small enough such that
        \[
            \left(\frac{3}{\sqrt{\e_1}}\right)^{2c_1n} \left(1 - 0.3 \frac{\e_2}{1 + \frac{4}{c_2} - \frac{4}{c_2n^{1-\d}}}\right)^{(1 - \d')n} < \left(1 - 0.2 \frac{\e_2}{1 + \frac{4}{c_2} - \frac{4}{c_2n^{1-\d}}}\right)^{(1 - \d')n}
        \]
        and
        \[
            \exp\left(\frac{c_1(1-\ln c_1)}{1-\d'}\right) \cdot \left(1 - 0.2 \frac{\e_2}{1 + \frac{4}{c_2} - \frac{4}{c_2n^{1-\d}}}\right) < 1 - 0.1 \frac{\e_2}{1 + \frac{4}{c_2} - \frac{4}{c_2n^{1-\d}}}.
        \]
\end{enumerate}
This proves Theorem \ref{msparse}.

\section{Incompressible Vectors} \label{ic}
Finally, we turn our attention to incompressible vectors. Let $Y_1$ denote the first column of $A$.
\begin{thm}
    \label{incomp}
    There exists a constant $c > 0$ such that for sufficiently large $n$ and any $\eta \in \cS_{IC}$ we have
    \[
        \bP\left[| \langle Y_1, \eta \rangle | < t \right] \le e^{-cn^\d} + O(t^2n^{1-\d}).
    \]
\end{thm}
\begin{proof}
    Notice there exists at least $\l_0n$ values of $i$ such that
    \[
        |\eta_{i}|^2 \geq \frac{1}{\l_1n}
    \]
    for some contants $\l_0, \l_1 > 0$ depending on $c_1, \e_1$. Let $S$ denote the set of such indices $i$. From the multiplicative Chernoff bound, we obtain
    \[
        \bP\left[\left|\left\{i \in S : Y_i \neq 0\right\}\right| \le \frac{1}{2}\l_0n^\d \right] \le \exp\left(-\frac{1}{8}\l_0n^\d\right).
    \]
    Let us now condition on the event that
    \[
        \left|\left\{i \in S : Y_i \neq 0\right\}\right| > \frac{1}{2}\l_0n^\d.
    \]
    Then by Lemma \ref{cgaussian} we conclude
    \[
        \bP\left[| \langle Y_1, \eta \rangle | < t \right] \le \frac{2\l_1}{\l_0}t^2n^{1-\d}.
    \]
    Theorem \ref{incomp} then follows by a union bound.
\end{proof}

\section{Proof of Theorem \ref{mainwnorm}} \label{mainwnormproof}
We are now ready to give a proof of Theorem \ref{mainwnorm}.
\begin{proof}
    Note we can rewrite
    \[
        \begin{split}
            \bP\left[\s_n(A) \le \e \text{ and } \|A\| \le Kn^{\frac{\d}{2}}\right] &= \bP\left[\inf_{x \in \cS^{n-1}}\|Ax\|_2 \le \e \text{ and } \|A\| \le Kn^{\frac{\d}{2}}\right] \\
                                                                                    &\le \sum_{i \in \{HC, MC, IC\}} \bP\left[\inf_{x \in \cS_i}\|Ax\|_2 \le \e \text{ and } \|A\| \le Kn^{\frac{\d}{2}}\right].
        \end{split}
    \]
    After applying Theorems \ref{hsparse} and \ref{msparse}, to prove Theorem \ref{mainwnorm} it suffices to show
    \[
        \bP\left[\inf_{x \in \cS_{IC}}\|Ax\|_2 \le \e \text{ and } \|A\| \le Kn^{\frac{\d}{2}}\right] \le e^{-cn^{\d}} + O(\e^2n^{2 - \d})
    \]
    for some $c > 0$ and sufficiently large $n$. Let $E$ denote the event that there exists $x \in \cS_{IC}$ such that
    \[
        \|Ax\|_2 \le \e, \quad \|A\| \le Kn^{\frac{\d}{2}}.
    \]
    Let $Y_i$ denote the $i$-th column of $A$ and let $W_i$ denote the span of the $Y_j$'s for $j \neq i$. Observe that
    \[
        |x_i|\dist(Y_i, W_i) \le \e.
    \]
    As $x \in \cS_{IC}$, there exists at least $\l_0n$ values of $i$ such that
    \[
        |x_i|^2 \ge \frac{1}{\l_1n}
    \]
    for some constants $\l_0, \l_1 > 0$ depending on $c_1, \e_2$. Notice such condition implies that $\dist(Y_i, W_i) \le \e\sqrt{\l_1n}$. Let $E_i$ denote the event that both this and $E$ occur, and note that
    \[
        \bP(E) \le \frac{1}{\l_0n} \sum_i \bP[E_i] = \frac{1}{\l_0} \bP[E_1]
    \]
    by double counting.

    As $| \langle Y_1 \cdot \eta_1 \rangle| \le \dist(Y_1, W_1)$ for $\eta_1 \perp W_1$, $\eta_1 \in \cS^{n-1}$, we conclude that
    \begin{align*}
        \bP(E_1) \le \bP\Big[&\forall \eta_1 \in \cS^{n-1} \text{ s.t. } B\eta_1 = 0 \text{ we have } | \langle Y_1 \cdot \eta_1 \rangle| \le \e\sqrt{\l_1n} \\
                           &\text{and } \|B\|_{op} \le Kn^\frac{\d}{2}\Big].
    \end{align*}
    By applying Theorems \ref{hsparse} and \ref{msparse} but with $A$ replaced with $B$, we get
    \[
        \bP\left[\inf_{\eta \in \cS^{n-1} - \cS_{IC}} \|B\eta\|_2 = 0 \text{ and } \|B\|_{op} \le Kn^\frac{\d}{2}\right] \le e^{-cn^\d}.
    \]
    for some $c > 0$ and sufficiently large $n$. (The same arguments work for $(n-1) \times n$ matrices.) Therefore, it suffices to bound
    \[
        \bP\left[\forall \eta \in \cS_{IC} \text{ we have } | \langle Y_1 \cdot \eta \rangle| \le \e\sqrt{\l_1n} \right].
    \]
    However, if $\eta \in \cS_{IC}$, then Theorem \ref{incomp} implies
    \[
        \bP\left[| \langle Y_1 \cdot \eta \rangle| \le \e\sqrt{\l_1n} \right] \le e^{-cn^\d} + O(\e^2n^{2-\d})
    \]
    for some $c > 0$ and sufficiently large $n$. Theorem \ref{mainwnorm} follows.
\end{proof}

\section{Operator Norm Bound} \label{mainproof}
To prove Theorem \ref{main}, it remains to prove Theorem \ref{opbound}. To do this, we will use a result that follows from Theorem 1.7 in \cite{BaRu17}:
\begin{lem}
    \label{ruop}
    Let $A'$ be the $n \times n$ random matrix in Theorem \ref{main} but with
    \[
        \xi_{ij} \sim \cN_\bR(0,1), \quad \d_{ij} \overset{d}{=} \d_{n^{\d-1}}.
    \]
    Then there exists constants $K', c' > 0$ such that for sufficiently large $n$ we have
    \[
        \bP[\|A'\|_{op} \ge K'n^{\frac{\d}{2}}] \le e^{-c'n^\d}.
    \]
\end{lem}
Write
\[
    A = \Re(A) + i\Im(A)
\]
where $\Re(A)$ denotes the real part of $A$ and $\Im(A)$ denotes the imaginary part of $A$. Observe $\Re(A), \Im(A)$ are i.i.d. random matrices. Furthermore, from Lemma \ref{ruop} there exists $K', c' > 0$ such that
\begin{align*}
    & \bP\left[\|\Re(A)\|_{op} \ge \frac{K'}{\sqrt{2}}n^{\frac{\d}{2}}\right] \le e^{-c'n^\d} \\
    & \bP\left[\|\Im(A)\|_{op} \ge \frac{K'}{\sqrt{2}}n^{\frac{\d}{2}}\right] \le e^{-c'n^\d}.
\end{align*}
As $\|A\|_{op} \le \|\Re(A)\|_{op} + \|\Im(A)\|_{op}$, the union bound gives us
\[
    \bP\left[\|A\|_{op} \ge \sqrt{2} K'n^{\frac{\d}{2}}\right] \le 2e^{-c'n^\d},
\]
proving Theorem \ref{opbound} and thus Theorem \ref{main}.

\section{Discussion On Sparse Regularization}
When a matrix has undesirable properties (such as a small least singular value), it is often useful to add a random regularizing perturbation. That is, given a matrix $M \in \bC^{n \times n}$ one often studies $M + \l A$, where $A$ is a random matrix with independent entries and $\l > 0$ is a small constant. (For an example of this technique in action, see \cite{ST04}.)

In many cases, $A$ is chosen from the complex Ginibre ensemble. The original motivation of this paper was the following question:
\begin{quote}
    What happens if the perturbation $A$ is a random sparse matrix?
\end{quote}
Theorem \ref{main} states that for a certain type of random sparse matrix $A$, we have
\[
    \bP[\s_n(0 + \l A) \le \e] \le e^{-cn^\d} + O_n\left(\frac{\e^2}{\l^2}\right).
\]
Intuitively, as $M = 0$ is ``as singular as a matrix can be," one may hope that $M = 0$ is a worst-case scenario and be able to conclude that
\[
    \bP\left[\s_{\text{n}}(M + \l A) \le \e\right] \le e^{-cn^\d} + O_n\left(\frac{\e^2}{\l^2}\right), \quad \forall M \in \bC^{n \times n}.
\]
Indeed, such shift-independence is true for $A$ drawn from the Ginibre ensemble. However, in this section, we will show that Theorem \ref{main} is not shift-independent.

In fact, we prove a stronger statement:
\begin{thm}
    \label{shift}
    Suppose $A$ is a $n \times n$ random matrix with independent entries where each entry has variance $< K$. If $\bP(a_{ij} = 0) = P > 0$ for any $a_{ij}$, then there exists a matrix $M$ with arbitrarily large $t = \|M\|_{op}$ such that
    \[
        \bP\left[\s_n(M + \l A) \le \frac{C\l^2n^{\frac{3}{2}}}{t} \right] \ge 0.99P
    \]
    where $C > 0$ depends only on $P$ and $K$.
\end{thm}
\begin{remark}
    In Theorem \ref{shift}, we do not require $A$ to have identically distributed entries. Also the power of $n$ has not been optimized.
\end{remark}
Observe Theorem \ref{shift} applies if \textit{any} entry of $A$ has a nonzero chance of being 0. This includes having a random matrix with i.i.d. $\cN_\bC(0, 1)$ entries, but with just \textit{one} entry randomly set to 0. 

\begin{proof}
    Recall that
    \[
        \s_{n}(M + \l A) = \inf_{x \in \cS^{n-1}} \|(M + \l A)x\|_2
    \]
    where $\cS^{n-1} \subset \bC^n$ denotes the unit sphere. For any deterministic unitary matrices $U$ and $V$, observe that
    \begin{align*}
        \inf_{x \in \cS^{n-1}} \|(M + \l A)x\|_2 &= \inf_{x \in \cS^{n-1}} \|V^*(M + \l A)Ux\|_2 \\
                                                 &= \inf_{x \in \cS^{n-1}} \|V^*MU + \l V^*AUx\|_2.
    \end{align*}
    By an appropriate choice of $U$ and $V$ to permute the rows and columns of $A$, we can set
    \[
        (V^*AU)_{nn} = a_{ij}.
    \]
    Consequently, it suffices to prove Theorem \ref{main} with $i = j = n$.

    Let us define
    \[
        M =
        \begin{pmatrix}
            t & 0 & 0 & \dots & 0 & 0\\
            0 & t & 0 & \dots & 0 & 0\\
            0 & 0 & t & \dots & 0 & 0\\
            0 & 0 & 0 & \ddots & t & 0\\
            0 & 0 & 0 & \dots & 0 & 0\\
        \end{pmatrix}.
    \]
    In other words, $M = t \cdot \text{Id}$ but with the lower right corner entry zeroed out. Let us also define
    \[
        x =
        \begin{pmatrix}
            -\frac{\l}{t}a_{1n} \\
            -\frac{\l}{t}a_{2n} \\
            \vdots \\
            -\frac{\l}{t}a_{(n-1)n} \\
            1
        \end{pmatrix}.
    \]
    Observe now that
    \[
        Mx = 
        \begin{pmatrix}
            t & 0 & 0 & \dots & 0 & 0\\
            0 & t & 0 & \dots & 0 & 0\\
            0 & 0 & t & \dots & 0 & 0\\
            0 & 0 & 0 & \ddots & t & 0\\
            0 & 0 & 0 & \dots & 0 & 0\\
        \end{pmatrix}
        \begin{pmatrix}
            -\frac{\l}{t}a_{1n} \\
            -\frac{\l}{t}a_{2n} \\
            \vdots \\
            -\frac{\l}{t}a_{(n-1)n} \\
            1
        \end{pmatrix}
        = -\l
        \begin{pmatrix}
            a_{1n} \\
            a_{2n} \\
            \vdots \\
            a_{(n-1)n} \\
            0
        \end{pmatrix}.
    \]
    In the event that $a_{nn} = 0$ we have
    \begin{align*}
        \|(M + \l A)x\|_2 = \|Mx + \l Ax\|_2
        &=
        \left\|
        \l
        \begin{pmatrix}
            a_{11} & \dots & a_{1(n-1)}\\
            a_{21} & \dots & a_{2(n-1)}\\
            \vdots & \ddots & \vdots \\
            a_{n1} & \dots & a_{n(n-1)}\\
        \end{pmatrix}
        \begin{pmatrix}
            -\frac{\l}{t}a_{1n} \\
            -\frac{\l}{t}a_{2n} \\
            \vdots \\
            -\frac{\l}{t}a_{(n-1)n} \\
        \end{pmatrix}
        \right\|_2 \\
        &=
        \frac{\l^2}{t}
        \left\|
        \begin{pmatrix}
            a_{11} & \dots & a_{1(n-1)}\\
            a_{21} & \dots & a_{2(n-1)}\\
            \vdots & \ddots & \vdots \\
            a_{n1} & \dots & a_{n(n-1)}\\
        \end{pmatrix}
        \begin{pmatrix}
            a_{1n} \\
            a_{2n} \\
            \vdots \\
            a_{(n-1)n} \\
        \end{pmatrix}
        \right\|_2.
    \end{align*}
    A straightforward application of Markov's inequality yields
    \begin{align*}
        \left\|
        \begin{pmatrix}
            a_{11} & \dots & a_{1(n-1)}\\
            a_{21} & \dots & a_{2(n-1)}\\
            \vdots & \ddots & \vdots \\
            a_{n1} & \dots & a_{n(n-1)}\\
        \end{pmatrix}
        \right\|_{op} &\le
        \left\|
        \begin{pmatrix}
            a_{11} & \dots & a_{1(n-1)}\\
            a_{21} & \dots & a_{2(n-1)}\\
            \vdots & \ddots & \vdots \\
            a_{n1} & \dots & a_{n(n-1)}\\
        \end{pmatrix}
        \right\|_{\text{HS}} \le C'n, \\
        \left\|
        \begin{pmatrix}
            a_{1n} \\
            a_{2n} \\
            \vdots \\
            a_{(n-1)n} \\
        \end{pmatrix}
        \right\|_2 &\le C'\sqrt{n}
    \end{align*}
    with probability $1 - 0.01P$ for a sufficiently large $C'$ depending only on $K$ and $P$. Thus with probability $0.99P$ we have
    \[
        \s_{n} (M + \l A) \le \frac{\|(M + \l A)x\|_2}{\|x\|_2} \le C'^2n^\frac{3}{2}\frac{\l^2}{t},
    \]
    as desired.
\end{proof}
Consequently, a correct generalization of the lower tail estimate in Theorem \ref{main} to the shifted random matrix $M + \l A$ will depend on $M$ or some property of $M$ (perhaps, say, the norm $\|M\|_{op}$). Thus a natural follow-up question is to ask:
\begin{quote}
    How does $\bP[\s_n(M + \l A) < \e]$ vary with $M$?
\end{quote}

\section{Acknowledgements}
This work was done as part of my senior thesis at UC Berkeley. First and foremost, I would like to thank my wonderful advisor, Professor Nikhil Srivastava, who suggested this problem and whose guidance and encouragement made this all possible. I would also like to thank Professor Mark Rudelson for providing thoughtful feedback on an earlier version of this thesis.

\newpage
\bibliographystyle{alpha}
\bibliography{ref}

\end{document}